\documentclass[12pt,leqno]{article}
\usepackage{amsmath,amssymb,amsbsy,amsfonts,amsthm,latexsym,
         amsopn,amstext,amsxtra,euscript,amscd,amsthm,mathdots}

\usepackage{amssymb}

\newtheorem{theorem}{Theorem}[section]
\newtheorem{lemma}[theorem]{Lemma}
\newtheorem{corollary}[theorem]{Corollary}

\theoremstyle{definition}
\newtheorem{definition}[theorem]{Definition}
\newtheorem{example}[theorem]{Example}

\theoremstyle{remark}
\newtheorem{remark}[theorem]{Remark}

\numberwithin{equation}{section}

\def\Z{{\mathbb Z}}
\def\R{{\mathbb R}}

\begin{document}

\begin{large}
\centerline{\large \bf New upper bounds for the number of divisors function}
\end{large}
\vskip 10pt
\begin{large}
\centerline{\sc Jean-Marie De Koninck \mbox{\rm and } Patrick Letendre }
\end{large}
\vskip 10pt
\begin{abstract}
Let $\tau(n)$ stand for the number of divisors of the positive integer $n$. We obtain upper bounds for $\tau(n)$ in terms of $\log n$ and the number of distinct prime factors of $n$.
\end{abstract}
\vskip 10pt
\noindent AMS Subject Classification numbers: 11N37, 11N56

\noindent Key words: number of divisors function

\vskip 20pt

\section{Introduction and notation}

Let $\tau(n)$ stand for the number of divisors of the positive integer $n$ and $\omega(n)$ stand for the number of prime factors of the positive integer $n$. We shall also be using the functions
$$
\gamma(n):=\prod_{p\mid n}p, \qquad \beta(n):=\prod_{p\mid n} \frac 1{\log p}
.$$

In 1915, Ramanujan \cite{sr} obtained the inequality
\begin{equation}\label{raman-i}
\tau(n) \le \left(\frac{\log(n\gamma(n))}{\omega(n)}\right)^{\omega(n)}\beta(n)\qquad(n \ge 2).
\end{equation}

In this paper, we compute explicitly some interesting limit cases of \eqref{raman-i} and show that for $k=\omega(n)\ge 74$,
$$
\tau(n)<\left(1+\frac{\log n}{k\log k}\right)^k.
$$
We also provide a short proof of \eqref{raman-i} in Corollary \ref{cor:1}.

From here on, for each integer $k \ge 0$, we let
$$
n_k:= p_1p_2 \cdots p_k, \ \mbox{ the product of the first $k$ primes}\ (\mbox{with}\ n_0=1).
$$
Also, when we write $\log_+ x$, we mean $\log\max(2,x)$.

Finally, given the factorization of an integer $n=q_1^{\alpha_1} \cdots q_k^{\alpha_k}$ with $q_1 <\cdots < q_k$, we call the vector $(\alpha_1,\ldots,\alpha_k)$ the {\it exponent vector} of $n$.

\section{Background results}

It is well known that
\begin{equation}\label{jensen1}
2^{\omega(n)}\le \tau(n)\le \left(1+\frac{\Omega(n)}{\omega(n)}\right)^{\omega(n)} \qquad (n \ge 2),
\end{equation}
where $\Omega(n)$ stands for the number of prime divisors of $n$ counting their multiplicity. Here, the lower bound is best possible in general and the upper bound is of great interest. For instance, it is known that the quotient $\displaystyle{\frac{\Omega(n)}{\omega(n)}}$ is near $1$ for almost all integers $n$, as was shown for instance by the first author in \cite{kn:jmdk-1974}. In fact, one can use \eqref{jensen1} and the estimate
$$
|\{n\le x : \Omega(n) \ge \varkappa \omega(n)\}|\ll x(\log\log x)(\log x )^{2^{1-\varkappa}-1}
$$
valid for all $\varkappa \ge 1$ and $x\ge 3$ (see Corollary $3.6$ p.436 in Tenenbaum \cite{gt} or for an even sharper estimate, Balazard \cite{mb}) to show that for every fixed $\varepsilon>0$,
$$
\tau(n)\le (2+\varepsilon)^{\omega(n)} \qquad \mbox{for almost all }n.
$$

We are motivated by the fact that, since Wigert \cite{sw}, we know that
$$
\log \tau(n) \le \frac{(\log 2)(\log n)}{\log \log n} + O \left(   \frac{\log n}{(\log \log n)^2} \right),
$$
and by the fact that it has been proved by Nicolas and Robin \cite{jln:gr} that the maximum value of the function
\begin{equation}\label{Nicolas-Robin}
n\mapsto\frac{\log(\tau(n))\log\log n}{\log 2 \log n}\qquad(n \ge 3)
\end{equation}
is attained at $n=6\,983\,776\,800=2^5 \cdot 3^3 \cdot 5^2 \cdot 7 \cdot 11 \cdot 13 \cdot 17 \cdot 19=720n_8$ and that its value is approximately  $1.5379$. Much more is known on the ratio \eqref{Nicolas-Robin}, as explained in \cite{jln}. But, meanwhile, those large values are almost never attained since it has been proved by Erd\H{o}s and Nicolas \cite{pe:jln} that, given any real $\vartheta \in (0,1)$, the cardinality of the set of those $n\le x$ for which
$$
\omega(n)\ge \vartheta\frac{\log x}{\log\log x}
$$
is $\displaystyle{ \ll x^{1-\vartheta+o(1)} }$ as $x\to \infty$. Furthermore, this set corresponds exactly to the set of values where $\tau(n)$ is large as can still be seen from \eqref{result}.

Before stating our main results, we introduce the function $\lambda(n)$ defined implicitly by 
$$
\tau(n)=\left(1+\frac{\lambda(n)\log n}{k\log k}\right)^k,
$$
where $k=\omega(n)\ge 2$. Therefore, for each integer $n\ge 2$ with $\omega(n)=k\ge 2$, we set
\begin{equation} \label{eq:def-lambda}
\lambda(n):=\frac{(\tau(n)^{1/k}-1)k\log k}{\log n}.
\end{equation}

\section{Main results}

\begin{theorem} \label{thm:1}
For every integer $n\ge 2$,
\begin{equation}\label{inequality1}
\tau(n)\le \left(\frac{\eta_2\log n}{\omega(n)\log_+ \omega(n)}\right)^{\omega(n)},
\end{equation}
where
$$
\eta_2:=\exp\biggl(\frac{1}{6}\log 96-\log\left(\frac{\log 60060}{6\log 6}\right)\biggr)= 2.0907132\dotso
$$
\end{theorem}

\begin{theorem} \label{thm:2}
For every integer $n>24n_{16}=782139803452561073520$,
\begin{equation}\label{inequality2}
\tau(n)\le \left(\frac{2\log n}{\omega(n)\log_+ \omega(n)}\right)^{\omega(n)}.
\end{equation}
Moreover, the inequality remains true for all $n\ge 2$ with $\omega(n)\le 3$.
\end{theorem}

\begin{theorem} \label{thm:3}
For every integer $n\ge 2$,
\begin{equation}\label{inequality3}
\tau(n)\le \left(1+\eta_3\frac{\log n}{\omega(n)\log_+ \omega(n)}\right)^{\omega(n)}
\end{equation}
where
$$
\eta_3:=\lambda(720n_7)=\frac{(1152^{1/7}-1)7\log 7}{\log 367567200}=1.1999953\dotso
$$
\end{theorem}

\begin{theorem} \label{thm:4}
For every positive integer $n$ with $k=\omega(n)\ge 74$,
\begin{equation}\label{result}
\tau(n)<\left(1+\frac{\log n}{k\log k}\right)^k.
\end{equation}
\end{theorem}

\begin{remark}
The number
$$
n'=2^{13}\cdot3^8\cdot5^5\cdot7^4\cdot11^3\cdot13^3\cdot17^3\cdot19^2\cdots53^2\cdot59\cdots367,
$$
whose prime factors are the first $73$ prime numbers, shows that Theorem \ref{thm:4} is best possible since $\lambda(n')=1.0008832\dotso$ In fact, one can find similar examples $n$ (that is, with $\lambda(n)>1$) for each $\omega(n)=k \in [3,73]$. Also, the methods used in the proof of Theorem \ref{thm:4} allow one to show that the largest value of $\lambda(n)$, with $\omega(n)=74$, is attained only by the number
$$
n''=2^{13}\cdot3^8\cdot5^5\cdot7^4\cdot11^3\cdot13^3\cdot17^3\cdot19^2\cdots53^2\cdot59\cdots373
$$
and for which $\lambda(n'')=0.99991077\dotso$ (Observe that the number $n'$ realizes the unique maximum of the function $\lambda$ among the integers $n$ with exactly $73$ distinct prime factors.)

By comparing the lower bound in \eqref{jensen1} with \eqref{result} and after some computation, one can show that the inequality
$$
n \ge \omega(n)^{\omega(n)} \qquad (n \ge 2)
$$
holds for each $n$ satisfying $\omega(n) \notin [4,12]$ or $n >43n_{11}$. This helps to understand why Theorem \ref{thm:4} is more powerful than Theorem \ref{thm:2}.
\end{remark}

\begin{theorem} \label{thm:5}
The largest integer $n$ with $k=\omega(n) \ge 44$ for which $\lambda(n) \ge 1$ is the integer made up of the first $44$ primes that has the exponent vector
\begin{eqnarray}\label{number}
\\ \nonumber
& (354, 223, 152, 125, 102, 95, 86, 83, 77, 72, 71, 67, 65, 64, 63, 61, 59, 59, 57, 57, 56, \\ \nonumber
& \hskip 5pt 55, 55, 54, 53, 52, 52, 52, 51, 51, 50, 49, 49, 49, 48, 48, 48, 47, 47, 47, 46, 46, 46, 46). \\ \nonumber
\end{eqnarray}
\end{theorem}

\section{Preliminary lemmas}

\begin{definition}\label{def_1}
Let $x_i$, with $i \in \{1,\dotso,k\}$, be fixed real numbers that satisfy $0 < x_1 \le \cdots \le x_k$. Let
$$
\mu:=\frac{x_1+\cdots+x_k}{k}
$$
and
$$
\varpi:=\sum_{i=1}^{k}|x_i-\mu|.
$$
Assume also that $x_1 \le \cdots \le x_m \le \mu \le x_{m+1} \le \cdots \le x_k$ for a fixed $m \in \{1,\dotso,k-1\}$ where $k \ge 2$. Further set
$$
\mu_1:=\frac{x_1+\cdots+x_m}{m}=\mu-\frac{\varpi}{2m},
$$
$$
\mu_2:=\frac{x_{m+1}+\cdots+x_k}{k-m}=\mu+\frac{\varpi}{2(k-m)}
$$
and also
$$
\varpi_1:=\sum_{i=1}^{m}|x_i-\mu_1|
$$
and
$$
\varpi_2:=\sum_{i=m+1}^{k}|x_i-\mu_2|.
$$
\end{definition}

\begin{example}
Here it is how this notation will be used throughout the proof of Theorem 5.  Let's fix an integer $n=q_1^{\alpha_1}\cdots q_k^{\alpha_k}$. For each $i\in \{1,2,\ldots,k\}$, we define $\theta_i$ implicitly by $n^{\theta_i}=q_i^{\alpha_i}$, so that $\theta_1+\cdots+\theta_k=1$. We write
\begin{equation}\label{x}
x_i:=\frac{(\alpha_i+1)\log q_i}{\log n}
\end{equation}
and assume that the primes $q_i$ are ordered in such a way that \eqref{museq} holds. In this case we have
\begin{equation}\label{def:mu}
\mu=\frac{1}{k}\left(1+\frac{\log \gamma(n)}{\log n}\right),
\end{equation}
$$
\varpi=\frac{1}{k}\sum_{i=1}^{k}\left|\frac{(\alpha_i+1)k\log q_i-\log \gamma(n)}{\log n}-1\right| =:\frac{\varpi^\prime}{k},
$$
$$
\mu_1=\frac{1}{k}\left(1+\frac{\log \gamma(n)}{\log n}\right)-\frac{\varpi^\prime}{2km},\quad\mu_2=\frac{1}{k}\left(1+\frac{\log \gamma(n)}{\log n}\right)+\frac{\varpi^\prime}{2k(k-m)},
$$
\begin{equation}\label{w1}
\varpi_1=\frac{1}{k}\sum_{i=1}^{m}\left|\frac{(\alpha_i+1)k\log q_i-\log \gamma(n)}{\log n}-1+\frac{\varpi^\prime}{2m}\right| =:\frac{\varpi_1^\prime}{k}
\end{equation}
and
\begin{equation}\label{w2}
\varpi_2=\frac{1}{k}\sum_{i=m+1}^{k}\left|\frac{(\alpha_i+1)k\log q_i-\log \gamma(n)}{\log n}-1-\frac{\varpi^\prime}{2(k-m)}\right| =:\frac{\varpi_2^\prime}{k}.
\end{equation}
\end{example}

\begin{lemma}\label{lem:01}
{\bf (i)} For $k \ge 1$ we have
\begin{equation} \label{je1}
x_1\cdots x_k \le \mu^k.
\end{equation}
{\bf (ii)} For $k \ge 2$ we have
\begin{equation} \label{je2}
x_1\cdots x_k \le \mu_1^m\mu_2^{k-m}.
\end{equation}
{\bf (iii)} For $k \ge 4$, let $m \in \{2,\dots,k-2\}$, $m_1 \in \{1,\dots,m-1\}$, $m_2 \in \{1,\dots,k-m-1\}$ and assume that
\begin{equation}\label{museq}
{\scriptstyle
 0 < x_1 \le \dots \le x_{m_1} \le \mu_1 \le x_{m_1+1} \le \dots \le x_m \le \mu \le x_{m+1} \le \dots \le x_{m+m_2} \le \mu_2 \le x_{m+m_2+1} \le \dots \le x_k}.
\end{equation}
Then,
\begin{equation} \label{je3} {\scriptstyle
x_1 \cdots x_k \le \left(\mu_1-\frac{\varpi_1}{2m_1}\right)^{m_1}\left(\mu_1+\frac{\varpi_1}{2(m-m_1)}\right)^{m-m_1}\left(\mu_2-\frac{\varpi_2}{2m_2}\right)^{m_2}\left(\mu_2+\frac{\varpi_2}{2(k-m-m_2)}\right)^{k-m-m_2}.}
\end{equation}
\end{lemma}

\begin{proof}
In each case, we simply use the arithmetic-geometric inequality for the corresponding sub-product of variables for which we know the average.
\end{proof}

\begin{lemma}\label{ineqfond}
Let $k\ge 1$ be an integer, $z_i>0$ and $x_i\ge \frac{-1}{z_i}$ be real numbers for $i=1,\dotso,k$, and assume that
$$
x_1+\cdots+x_k=1.
$$
Then,
\begin{equation}\label{jensen3}
\prod_{i=1}^k (1+x_iz_i)\le \prod_{i=1}^{k}\left(\frac{z_i}{k}\right)\biggl(1+\sum_{i=1}^k \frac 1{z_i}\biggr)^k,
\end{equation}
with equality if and only if
$$
x_i=\frac 1k \biggl(1+\sum_{j=1}^k \frac 1{z_j}\biggr)-\frac{1}{z_i} \qquad (i=1,\ldots,k).
$$
\end{lemma}

\begin{proof}
Using the arithmetic geometric mean inequality, the hypothesis $z_i>0$ and the fact that for each $i$ we have $1+x_iz_i\ge 0$, we can write
\begin{eqnarray*}
\prod_{i=1}^k (1+x_iz_i) &= & \prod_{i=1}^k z_i\prod_{j=1}^k \left(x_j+\frac 1{z_j}\right)\\
&\le & \prod_{i=1}^k \left(\frac{z_i}{k}\right)\biggl(\sum_{j=1}^{k} \left(x_j+\frac{1}{z_j}\right)\biggr)^k\\
&= & \prod_{i=1}^{k}\left(\frac{z_i}{k}\right)\biggl(1+\sum_{j=1}^{k}\frac{1}{z_j}\biggr)^{k}.
\end{eqnarray*}
We have equality if and only if
$$
x_i+\frac{1}{z_i}=\frac{1}{k}\biggl(1+\sum_{j=1}^{k}\frac{1}{z_j}\biggr) \qquad (i=1,\ldots,k),
$$
thus completing the proof.
\end{proof}

\begin{corollary} \label{cor:1}
Assume the above notation. Then, for every integer $n\ge 2$,
\begin{equation}\label{fond1}
\tau(n)\le \left(\frac{\log n}{\omega(n)}\right)^{\omega(n)}\left(1+\frac{\log \gamma(n)}{\log n}\right)^{\omega(n)}\beta(n)
\end{equation}
and
\begin{equation}\label{fond2}
\tau(n)\le \left(\frac{2\log n}{\omega(n)}\right)^{\omega(n)}\beta(n).
\end{equation}
\end{corollary}

\begin{proof}[Proof of Corollary \ref{cor:1}]
Using inequality \eqref{jensen3} we have
$$
\tau(n)=\prod_{i=1}^{k}(1+\alpha_i)=\prod_{i=1}^{k}\left(1+\frac{\theta_i\log n}{\log q_i}\right)\le \left(\frac{\log n}{k}\right)^{k}\left(1+\frac{\log \gamma(n)}{\log n}\right)^{k}\prod_{p\mid n}\frac{1}{\log p},
$$
which proves (\ref{fond1}). Since $\log \gamma(n)\le \log n$, inequality (\ref{fond2}) follows immediately from (\ref{fond1}).
\end{proof}

In any event, observe that it follows from Corollary \ref{cor:1} that
\begin{equation}\label{ineq5}
\lambda(n)\le \biggl(\prod_{p\mid n}\frac{\log k}{\log p}\biggr)^{1/k}+\frac{\log \gamma(n)}{\log n}\biggl(\prod_{p\mid n}\frac{\log k}{\log p}\biggr)^{1/k}-\frac{k\log k}{\log n}.
\end{equation}

\begin{lemma} \label{lem:02}
{\bf (i)} Assume that $\mu > 0$, $\varpi \ge 0$, $m \ge 1$, $k-m \ge 1$ and $\mu-\frac{\varpi}{2m} > 0$. Then, the function
\begin{equation}\label{g1}
\left(\mu-\frac{\varpi}{2m}\right)^m\left(\mu+\frac{\varpi}{2(k-m)}\right)^{k-m}
\end{equation}
decreases when $\varpi$ increases.

\noindent{\bf (ii)} Assume that $\mu > 0$, $\varpi_1 \ge 0$, $\varpi_2 \ge 0$, $m \ge 1$, $k-m \ge 1$, $m_1 \ge 1$, $m_2 \ge 1$, $m-m_1 \ge 1$, $k-m-m_2 \ge 1$, $\mu-\frac{\varpi}{2m}-\frac{\varpi_1}{2m_1} > 0$ and $\mu+\frac{\varpi}{2(k-m)}-\frac{\varpi_2}{2m_2} > 0$. Then, the function
\begin{eqnarray}\label{f}
f(\varpi) & := & {\scriptstyle\left(\mu-\frac{\varpi}{2m}-\frac{\varpi_1}{2m_1}\right)^{m_1}\left(\mu-\frac{\varpi}{2m}+\frac{\varpi_1}{2(m-m_1)}\right)^{m-m_1}} \\ \nonumber
& & {\scriptstyle\times\left(\mu+\frac{\varpi}{2(k-m)}-\frac{\varpi_2}{2m_2}\right)^{m_2}\left(\mu+\frac{\varpi}{2(k-m)}+\frac{\varpi_2}{2(k-m-m_2)}\right)^{k-m-m_2}}
\end{eqnarray}
has the property that if $f^\prime(\varpi_0) < 0$ for some $\varpi_0 > 0$, then $f(\varpi) < f(\varpi_0)$ for each $\varpi > \varpi_0$.

\noindent{\bf (iii)} Assume that $A > 0$, $B > 0$, $C > 0$, $z > 0$, $\gamma_1 \ge 0$, $\gamma_2 \ge 0$, $\varrho_1 \ge 0$, $\varrho_2 \ge 0$, $\varrho_1+\varrho_2 = 1$ and $C < AB$. Then, the expression
\begin{equation} \label{z2}
B\left(\gamma_1+\frac{A}{z}\right)^{\varrho_1}\left(\gamma_2+\frac{A}{z}\right)^{\varrho_2}-\frac{C}{z}
\end{equation}
decreases when $z$ increases.

\noindent{\bf (iv)} Assume that $A > 0$, $B > 0$, $C > 0$, $z > 0$, $\gamma_1 \ge 0$, $\gamma_2 \ge 0$, $\gamma_3 \ge 0$, $\gamma_4 \ge 0$, $\varrho_1 \ge 0$, $\varrho_2 \ge 0$, $\varrho_3 \ge 0$, $\varrho_4 \ge 0$, $\varrho_1+\varrho_2+\varrho_3+\varrho_4= 1$ and $C < AB$. Then, the expression
\begin{equation} \label{z4}
B\left(\gamma_1+\frac{A}{z}\right)^{\varrho_1}\left(\gamma_2+\frac{A}{z}\right)^{\varrho_2}\left(\gamma_3+\frac{A}{z}\right)^{\varrho_3}\left(\gamma_4+\frac{A}{z}\right)^{\varrho_4}-\frac{C}{z}
\end{equation}
decreases when $z$ increases.
\end{lemma}

\begin{proof}
{\bf (i)} Since the function \eqref{g1} is assumed to be positive, it follows that its derivative with respect to $\varpi$ has the same sign than its logarithmic derivative with respect to $\varpi$. Then, since the logarithmic derivative is
$$
\frac{-1}{2\mu-\frac{\varpi}{m}}+\frac{1}{2\mu+\frac{\varpi}{k-m}},
$$
it is clearly strictly negative when $\varpi > 0$.

\noindent{\bf (ii)} Again, the function $f$ is assumed to be positive in which case its derivative with respect to $\varpi$ has the same sign than its logarithmic derivative with respect to $\varpi$. Also, we have
\begin{eqnarray*}
\left(\frac{f^\prime(\varpi)}{f(\varpi)}\right)^\prime & = & {\scriptstyle\frac{-m_1}{m^2\bigl(2\mu-\frac{\varpi}{m}-\frac{\varpi_1}{m_1}\bigr)^2}- \frac{m-m_1}{m^2\bigl(2\mu-\frac{\varpi}{m}+\frac{\varpi_1}{m-m_1}\bigr)^2}}\\
&& {\scriptstyle-\frac{m_2}{(k-m)^2\bigl(2\mu+\frac{\varpi}{k-m}-\frac{\varpi_2}{m_2}\bigr)^2} -\frac{k-m-m_2}{(k-m)^2\bigl(2\mu+\frac{\varpi}{k-m}-\frac{\varpi_2}{k-m-m_2}\bigr)^2}}
\end{eqnarray*}
which is clearly negative. We deduce that if $f^\prime(\varpi_0) < 0$ for some $\varpi_0 > 0$, then $f^\prime(\varpi) < 0$ for each $\varpi > \varpi_0$ which in turn implies
$$
f(\varpi)-f(\varpi_0)=\int_{\varpi_0}^{\varpi}f^\prime(t)dt<0
$$
for each $\varpi > \varpi_0$, thus establishing our claim.

\noindent{\bf (iii)} We take the derivative of \eqref{z2} with respect to $z$ and multiply by $z^2$. We then see that the wanted property is equivalent to
\begin{equation}\label{CAB}
C < AB\left(\gamma_1+\frac{A}{z}\right)^{\varrho_1}\left(\gamma_2+\frac{A}{z}\right)^{\varrho_2}\biggl(\frac{\varrho_1}{\gamma_1+\frac{A}{z}}+\frac{\varrho_2}{\gamma_2+\frac{A}{z}}\biggr).
\end{equation}
Now, from Jensen's inequality for the exponential function, we have
$$
\frac{1}{z_1^{\varrho_1}z_2^{\varrho_2}} \le \frac{\varrho_1}{z_1}+\frac{\varrho_2}{z_2} \qquad (z_1,z_2>0).
$$
We deduce that the hypothesis $C < AB$ implies \eqref{CAB}. {\bf (iv)} is done in the same manner and the proof is complete.
\end{proof}

\begin{lemma} \label{lem:03}
Let $A$ and $B$ be fixed positive real constants. Consider the function $\psi:=\Z\times\R^{*}\times\R\rightarrow\R_{\ge0}$ defined by
\begin{equation}\label{fonc1}
\psi(\alpha,x,\varphi)=\left|\frac{(\alpha+1)B-A}{x}-\varphi\right|.
\end{equation}
{\bf (i)} Assume that $x_1, \varphi_1> 0$. The minimum of the function $\psi(\alpha,x,\varphi)$ for $\alpha \in \Z$, $x \in [x_1,x_2]$ and $\varphi \in [\varphi_1,\varphi_2]$ is either $0$ or is given by the minimum over the eight possibilities provided by
$$
\alpha\in\left\{\left\lfloor\frac{x_2\varphi_2+A}{B}\right\rfloor-1,\left\lceil\frac{x_1\varphi_1+A}{B}\right\rceil-1\right\},\quad x \in \{x_1,x_2\}\quad\mbox{and}\quad \varphi \in \{\varphi_1,\varphi_2\}.
$$
The minimum is $0$ if and only if
\begin{equation}\label{equi1}
\left\lceil\frac{x_1\varphi_1+A}{B}\right\rceil \le \left\lfloor\frac{x_2\varphi_2+A}{B}\right\rfloor.
\end{equation}
\noindent{\bf (ii)} Let $\delta > 0$ be a fixed real number and assume that $x_1, \varphi_1> 0$. The minimum of the function  $\psi(\alpha,x,1)$ for $x \in [x_1,x_2]$ and $\alpha \in \Z\backslash{\scriptstyle\left\{\left\lceil\frac{(1-\delta)x_1+A}{B}\right\rceil-1,\dots,\left\lfloor\frac{(1+\delta)x_2+A}{B}\right\rfloor-1\right\}}$ is given by the minimum over the four possibilities provided by
$$
\alpha\in\left\{\left\lceil\frac{(1-\delta)x_1+A}{B}\right\rceil-2,\left\lfloor\frac{(1+\delta)x_2+A}{B}\right\rfloor\right\}\quad\mbox{and}\quad x \in \{x_1,x_2\}.
$$
\end{lemma}

\begin{proof}
{\bf (i)} First, assume that the minimum is $0$. Choose $(\alpha,x,\varphi)$ that realizes $0$. We deduce that $\alpha+1=\frac{x\varphi+A}{B}$ and then it is equivalent to having \eqref{equi1}. Now, assume that the minimum is not zero. In this case, $\left\lfloor\frac{x_2\varphi_2+A}{B}\right\rfloor < \left\lceil\frac{x_1\varphi_1+A}{B}\right\rceil = \left\lfloor\frac{x_2\varphi_2+A}{B}\right\rfloor+1$. Also, if $(\alpha,x,\varphi)$ realizes the minimum then there are two cases. We either have $\frac{(\alpha+1)B-A}{x}-\varphi \ge 0$, in which case $\alpha+1 \ge \left\lceil\frac{x_1\varphi_1+A}{B}\right\rceil$, or we have $\frac{(\alpha+1)B-A}{x}-\varphi \le 0$ , in which case we have  $\alpha+1 \le \left\lfloor\frac{x_2\varphi_2+A}{B}\right\rfloor$. It is then clear that the minimum is attained for $\alpha \in \left\{\left\lfloor\frac{x_2\varphi_2+A}{B}\right\rfloor-1,\left\lceil\frac{x_1\varphi_1+A}{B}\right\rceil-1\right\}$. To conclude, we remark that, once $\alpha$ is fixed, $\frac{(\alpha+1)B-A}{x}-\varphi$ attained its extremum at the edges of the intervals since it is a sum of independent monotone functions.

\noindent{\bf (ii)} The choice for $\alpha$ is clear. Also, if we assume that the minimum is not $0$ then the choice for $x$ is also clear. Now, assume the contrary, that the minimum is $0$ and that it is attained at $(\alpha,x)$ with $\alpha=\left\lceil\frac{(1-\delta)x_1+A}{B}\right\rceil-2=\frac{(1-\delta)x_1+A}{B}-2+\xi$ for some $\xi \in [0,1]$. In this case,
$$
{\scriptstyle\frac{(\alpha+1)B-A}{x}=\frac{\left(\frac{(1-\delta)x_1+A}{B}-2+\xi+1\right)B-A}{x}=\frac{(1-\delta)x_1+(\xi-1)B}{x} \le 1-\delta<1}
$$
and similarly for the other choice of $\alpha$. This shows that the minimum is not $0$ and the proof is complete.
\end{proof}

\begin{lemma}  \label{lem:2}
We have
\begin{equation}\label{ineq3}
\sum_{i=1}^{k}\log p_i \le k(\log k+\log\log k-1/2)\qquad\mbox{for }\ k\ge 5,
\end{equation}

\begin{equation}\label{ineq4}
\sum_{i=1}^{k}\log\log p_i\ge k\left(\log\log k+\frac{\log\log k-3/2}{\log k}\right)\qquad\mbox{for}\ k\ge 6
\end{equation}
and
\begin{equation}\label{ineq2}
\beta(n_k) \le (\log k)^{-k} \qquad\mbox{for}\ k\ge 44.
\end{equation}
\end{lemma}

\begin{proof}
We first prove inequality (\ref{ineq3}) using induction. First observe that the inequality holds for $k=5$. Assuming that the inequality holds for some $k\ge 5$, we will show that it must then hold for $k+1$. Since $p_j< \frac 32 j\log j$ for each $j\ge 4$, it follows that
\begin{eqnarray*}
\sum_{i=1}^{k+1}\log p_i& \le &  k(\log k+\log\log k-\frac{1}{2})+\log p_{k+1}\\
& < & (k+1)(\log(k+1)+\log\log(k+1)-\frac{1}{2})+\frac{1}{2}+\log \frac{3}{2}-k\log(1+\frac{1}{k})\\
& < & (k+1)(\log(k+1)+\log\log(k+1)-\frac{1}{2}),
\end{eqnarray*}
where this last inequality holds because of the fact that
$$
\log \left(1+\frac{1}{k}\right)\ge \frac{1}{k}-\frac{1}{2k^2} \qquad \mbox{for all }k\ge 1
$$
implies that
$$
1/2+\log\frac{3}{2}-k\log(1+1/k)\le 1/2+\log\frac{3}{2}-1+\frac{1}{2k}<0  \qquad \mbox{for }k\ge 6,
$$
thus completing the proof of (\ref{ineq3}).

To prove inequality (\ref{ineq4}), we first verify using a computer that it holds for each $k\in[6,200\,000]$. For $k\ge 200\,001$, we proceed by induction. Since the function $\displaystyle{\frac{\pi(x)\log x}{x}}$ attains its maximum at $x=113$ with the value $1/c:=1.255\dotso$ (see for instance Rosser and Schoenfeld \cite{kn:rosser}), it follows that $p_j\ge  cj\log j$. It follows that
\begin{eqnarray*}
& & \sum_{i=1}^k \log \log p_i\\
& & \quad > k\left(\log\log k+\frac{\log\log k-3/2}{\log k}\right)+\log\log(k+1)+\frac{\log\log(k+1)+\log c}{\log (k+1)}\\
& & \qquad-\frac{1}{2}\left(\frac{\log\log(k+1)+\log c}{\log(k+1)}\right)^2\\
&& =(k+1)\left(\log\log(k+1)+\frac{\log\log(k+1)-3/2}{\log(k+1)}\right)\\
&&\quad+\frac{3/2+\log c }{\log(k+1)}-\frac{1}{2}\left(\frac{\log\log(k+1)+\log c    }{\log(k+1)}\right)^2\\
&& \quad+k\left(\log\log k -\log\log(k+1)+\frac{\log\log k -3/2}{\log k }-\frac{\log\log(k+1)-3/2}{\log(k+1)}\right).
\end{eqnarray*}
It remains to show that the sum of these last three terms is positive. Using the mean value theorem and the fact that $\log\log \xi     >5/2$ for $\xi\ge 200\,000$, we find that it suffices to show that
$$
\frac{3/2+\log c    }{\log(k+1)}\ge \frac{1}{2}\left(\frac{\log\log(k+1)+\log c    }{\log(k+1)}\right)^2+\frac{k}{\xi\log \xi,     }
$$
for some $\xi\in(k,k+1)$. It is therefore enough to show that
\begin{equation} \label{eq:23}
3/2+\log c-1\ge \frac{1}{2}\frac{(\log\log(k+1)+\log c)^2}{\log(k+1)}+\frac{1}{k\log k}.
\end{equation}
Since each of the two terms on the right of (\ref{eq:23}) decreases as a function of $k$ for $k\ge 200\,000$ and since (\ref{eq:23}) is true for $k=200\,000$, it follows that (\ref{eq:23}) holds for all $k\ge 200\,000$, thereby completing the proof of (\ref{ineq4}).

Finally, \eqref{ineq2} follows from \eqref{ineq4} and an easy verification with a computer.
\end{proof}

Let us further introduce the function
\begin{equation} \label{def-t}
t(n):=\frac{\tau(n)^{1/k}}{\log n }\qquad(n\ge2).
\end{equation}

\begin{lemma} \label{lem:3} Let $n\ge 2$ be an integer, $2 \le k=\omega(n)$ and $p$ be a prime number. If $p^\alpha\|n$ with $\alpha \ge 2$, then
\begin{equation}\label{ratio1}
\frac{\lambda(n)}{\lambda(n/p)}\le \left(1+\frac{2}{k\alpha}\right)\left(1-\frac{\log p }{\log n}\right)
\end{equation}
and
\begin{equation}\label{ratio2}
\frac{t(n)}{t(n/p)}\le \left(1+\frac{1}{k\alpha}\right)\left(1-\frac{\log p }{\log n}\right).
\end{equation}
Also, for $\ell \in \{1,2\}$, we have
\begin{equation} \label{ratio3}
\left(1+\frac{\ell}{k\alpha}\right)\left(1-\frac{\log p }{\log n}\right)<1 \Longleftrightarrow p>n^{\frac{\ell}{\alpha k+\ell}}
\end{equation}
and
\begin{equation} \label{ratio4}
\alpha=\max\left(2,\left\lceil\frac{\ell}{k}\left(\frac{\log n}{\log p}-1\right)\right\rceil\right) \Longrightarrow \left(1+\frac{\ell}{k\alpha}\right)\left(1-\frac{\log p}{\log n}\right) < 1.
\end{equation}
\end{lemma}

\begin{proof} We write $n=p^\alpha m$, so that $(p,m)=1$ and therefore,
\begin{eqnarray*}
\frac{\lambda(n)}{\lambda(n/p)}&= & \frac{\tau(n)^{1/k}-1}{\tau(n/p)^{1/k}-1}\frac{\log n/p}{\log n}\\
&= & \left(1+\frac{\tau(n)^{1/k}-\tau(n/p)^{1/k}}{\tau(n/p)^{1/k}-1}\right)\left(1-\frac{\log p}{\log n}\right)\\
& = & \left(1+\frac{\tau(m)^{1/k}}{\tau(n/p)^{1/k}-1}((\alpha+1)^{1/k}-\alpha^{1/k})\right)\left(1-\frac{\log p}{\log n}\right)\\
&\le & \left(1+\frac{\tau(n/p)^{1/k}}{\tau(n/p)^{1/k}-1}\frac{1}{k\alpha}\right)\left(1-\frac{\log p}{\log n}\right),
\end{eqnarray*}
where the last inequality follows from the fact that
$$
(\alpha+1)^{1/k}-\alpha^{1/k} \le \sup_{\xi\in[\alpha,\alpha+1]}\frac{\xi^{1/k}}{k\xi}=\frac{\alpha^{1/k}}{k\alpha}.
$$
Since the function $z\rightarrow\frac{z}{z-1}$ is strictly decreasing for $z>1$, the result then follows from the fact that $\tau(n/p)^{1/k}\ge 2$. The proof of inequality \eqref{ratio2} is similar and the proofs of \eqref{ratio3} and \eqref{ratio4} follow from an easy computation.
\end{proof}

\begin{lemma}\label{lem:upsilon} For any real $z>1$ and integer $n=q_1^{\alpha_1}\cdots q_k^{\alpha_k}\ge 2$, let
\begin{equation}\label{def:upsilon}
\upsilon(n,z):=\log k \left(1+\frac{\log \gamma(n)}{\log z   }\right)\beta(n)^{1/k}-\frac{k\log k }{\log z   }.
\end{equation}
Then,
\begin{equation} \label{eq:101}
\frac {d}{dz} \upsilon(n,z) \le 0 \qquad (n\ge 2)
\end{equation}
with strict inequality if $\omega(n) \ge 2$. Also,
\begin{equation} \label{eq:102}
\upsilon(n_k,n_k)<1 \qquad (k\ge 95).
\end{equation}
\end{lemma}

\begin{proof}
To prove (\ref{eq:101}), we first observe that
$$
\frac {d}{dz} \upsilon(n,z) = -\log k \frac{\log \gamma(n)}{z\log^2 z}\beta(n)^{1/k}+\frac{k\log k }{z\log^2 z},
$$
which implies that (\ref{eq:101}) is equivalent to
$$
k\le \beta(n)^{1/k} \log \gamma(n),
$$
which itself is an immediate consequence of the arithmetic geometric mean inequality.

To prove (\ref{eq:102}),
we must show that
\begin{equation}\label{eq:toshow}
2<\left(\frac{k}{\log n_k}+\frac{1}{\log k }\right)\beta(n_k)^{-1/k}.
\end{equation}
Using inequalities \eqref{ineq4} and then \eqref{ineq3} we see that the right hand side of \eqref{eq:toshow} is

\begin{eqnarray*}
&& >\left(\exp\left(\log\log k +\frac{\log\log k -3/2}{\log k }\right)\right)\cdot \left(\frac{1}{\log k }+\frac{1}{\log k +\log\log k -1/2}\right)\\
&& =\left(\exp\left(\frac{\log\log k -3/2}{\log k }\right)\right) \cdot \left(1+\frac{\log k }{\log k +\log\log k -1/2}\right)\\
&& >\left(1+\frac{\log\log k -3/2}{\log k }\right) \cdot \left(1+\frac{\log k }{\log k +\log\log k -1/2}\right)\\
&& =2+\frac{\log\log k -3/2}{\log k }-\frac{1}{\log k +\log\log k -1/2} = 2+\xi_k,
\end{eqnarray*}
say. Since $\xi_k>0$ for all $n\ge 35\,807$, inequality (\ref{eq:toshow}) is proved for $k\ge 35\,807$. On the other hand, using a computer, one can easily check that (\ref{eq:102}) holds for each integer $k\in [95,35806]$, thus completing the proof of (\ref{eq:102}).
\end{proof}

\begin{lemma} \label{lem:5}
Let $\alpha\in (0,1)$, $c_1,c_2\in \R$ with $c_1>0$, $c_2 > 0$ and $I:=(c^{-1/\alpha}_1+c_2,\infty)$. Consider the function $g:I \to \R$ defined by
$$
g(z):=\frac{c_1(z-c_2)^\alpha-1}{z}.
$$
Then, $g$ attains its unique maximum at some point $z_0>c^{-1/\alpha}_1+c_2$.
\end{lemma}

\begin{proof}
Consider the function $h:I\rightarrow\R$ given by
\begin{equation}\label{def:h}
h(z):=z^2(z-c_2)^{1-\alpha}g^\prime(z)=c_1\alpha z-c_1(z-c_2)+(z-c_2)^{1-\alpha}.
\end{equation}
It follows from this that $h$ and $g^\prime$ have the same sign and the same zeros in $I$. Moreover, $h(\infty)=-\infty$. On the other hand,
$$
h'(z)=c_1(\alpha-1)+\frac{1-\alpha}{(z-c_2)^\alpha},
$$
in which case,
$$
h'(z)=0\quad \Longleftrightarrow \quad 1=c_1(z-c_2)^\alpha,
$$
which is impossible for $z\in I$. Now, because $h'(\infty)<0$, this means that $h'(z)<0$ for $z\in I$. Our second claim then follows from the fact that the maximum is in $I$.
\end{proof}

\section{Proof of Theorem \ref{thm:1}}

It is easy to verify that \eqref{inequality1} holds when $\omega(n)=1$. For any $n$ with $\omega(n) \ge 2$, we introduce the function $r(n)$ defined implicitly by
$$
\tau(n)=\left(\frac{e^{r(n)}\log n}{\omega(n)\log \omega(n)}\right)^{\omega(n)}.
$$
Hence, for any $n$ with $\omega(n) \ge 2$, we have
\begin{equation}\label{r}
r(n):=\frac{1}{\omega(n)}\left(\log \tau(n)-\omega(n)\log\left(\frac{\log n  }{\omega(n)\log \omega(n)}\right)\right).
\end{equation}
Observe that for $n_{*}:=60060=2^2 \cdot 3 \cdot 5 \cdot 7 \cdot 11 \cdot 13$ we have $r(n_{*})=0.737505\dotso=\log \eta_2$. We claim that $n_{*}$ is the only integer $n$ with $\omega(n) \ge 2$ that maximizes the function $r$ (this function is bounded, as it will become clear below). To prove it, we proceed by contradiction. Assume that, for some $k \ge 2$, there exists an integer $n' \neq n_{*}$ with $\omega(n')=k$ for which $\eqref{inequality1}$ is false and moreover that $r(n')$ is maximal. It is clear that the factorization of $n'$ takes the form
\begin{equation}\label{shape}
n'=\prod_{i=1}^{k}p_i^{\alpha_i} \qquad \mbox{with } \alpha_1 \ge \alpha_2 \ge \cdots \ge \alpha_k,
\end{equation}
where the $p_i$'s are the primes in ascending order.

Using \eqref{fond2} (from Corollary \ref{cor:1}) and \eqref{ineq2} (from Lemma \ref{lem:2}), one easily see that $r(n') < \log 2=0.693\dotso$ if $k \ge 44$ which is non sense since $r(n_{*})=0.737\dotso$ Thus we must have $k \le 43$. Now, it follows from (\ref{fond2}) that
\begin{equation}\label{i1}
\tau(n') \le \left( \frac{2\log n'}k \right)^k  \beta(n')  \le  \left( \frac{2\log n'}k \right)^k  \beta(n_k).
\end{equation}
Inserting \eqref{i1} in \eqref{r}, we then get
\begin{eqnarray*}
r(n') & \le & \frac{1}{k}\left(\log \beta(n_k)+k\log\left(\frac{2\log n'}{k}\right)-k\log\left(\frac{\log n'}{k\log k}\right)\right)\\
& = & \log 2+\log\log k+\frac{\log \beta(n_k)}{k},
\end{eqnarray*}
a quantity which depends only on $k$. On the other hand, using a computer reveals that $r(n')<\log \eta_2$ for each $k \in [2,3] \cup [25,43]$. This contradicts the choice of $n'$. Therefore we only need to consider the cases when $k \in \{4,\dotso,24\}$.

Now, inserting \eqref{fond1} in \eqref{r}, we have that
\begin{eqnarray*}
r(n') & \le & \frac{1}{k}\left(\log \beta(n_k)+k\log\left(\frac{\log n'}{k}\right)+k\log\left(1+\frac{\log n_k}{\log n'}\right)-k\log\left(\frac{\log n'  }{k\log k}\right)\right)\\
& = &\frac{\log \beta(n_k)}{k}+\log\log k+\log\left(1+\frac{\log n_k}{\log n'}\right)=r_1(n',k),
\end{eqnarray*}
where
\begin{equation}\label{def:r1}
r_1(z,k):=\frac{\log \beta(n_k)}{k}+\log\log k+\log\left(1+\frac{\log n_k}{\log z}\right).
\end{equation}
We observe that the function $r_1(z,k)$ decreases when $z$ increases. Thus, defining $z_k$ as the unique solution in $z$ of $r_1(z,k)=\log \eta_2$, we obtain that $n' \le z_k$ given that $\omega(n')=k$.

We now consider the function
\begin{equation}\label{def:u}
u(x):=\max_{\ell \ge 0}\{\ell:\ n_\ell \le x\}.
\end{equation}
Observe that, since $n'$ is of the form \eqref{shape}, $u(z_k/n_k)$ is an upper bound for the rank $j$ of the largest prime $p_j$ such that $p_j^2 \mid n'$. One may verify that for each $k \in \{4,\dotso,24\}$ we have $u(z_k/n_k) \le 3$ implying that $j \le 3$. Now, recalling the definition of $t(n)$ given in \eqref{def-t}, we may write
$$
r(n)=\log t(n)+\log(\omega(n)\log \omega(n)).
$$
Hence, for a fixed value of $k=\omega(n)$, it follows that $r(n)$ increases or decreases along with $t(n)$. Therefore, our hypothesis implies that $t(n')$ is maximal. Thus for each $j \in \{1,2,3\}$, using inequality \eqref{ratio2} and the maximality of $t(n')$, we can write
$$
1\le\frac{t(n')}{t(n'/p_j)} \le \left(1+\frac{1}{k\alpha}\right)\left(1-\frac{\log p_j }{\log n'}\right) \le \left(1+\frac{1}{k\alpha}\right)\left(1-\frac{\log p_j }{\log z_k}\right),
$$
and we obtain the desired contradiction if this last expression is less than $1$, which will happen if the integer $\alpha \ge 2$ satisfying $p_j^\alpha \| n'$ is large enough. Using \eqref{ratio4} we get an upper bound for each of the first three components in the exponent vector of $n'$. In fact, one may verify that, for each $k \in \{4,\dotso,24\}$,
$$
(4,2,2,\underbrace{1,\dotso,1}_{k-3})
$$
is an upper bound (in each of its coordinates) for the exponent vector of $n'$, implying that there are just a small number of cases to verify. After all the computations are done, we obtain a finite set of pairs $(n,r(n))$ including $(n_{*},r(n_{*}))$ and find that all the other pairs in this set satisfy $r(n) < r(n_{*})$. This contradicts the existence of $n'$ and completes the proof of Theorem \ref{thm:1}.

\section{Proof of Theorem \ref{thm:2}}

We first verify that \eqref{inequality2} does not hold for the integer
$$
n_{*}:=782139803452561073520=24n_{16}.
$$
If $n$ is any integer such that $\omega(n) \ge 44$, then it follows from Corollary \ref{cor:1} and Lemma \ref{lem:2} that inequality \eqref{inequality2} is satisfied (see the proof of Theorem \ref{thm:1}). Since it is clear that \eqref{inequality2} holds when $\omega(n)=1$, it remains only to consider the set of integers $n$ such that $2 \le \omega(n) \le 43$. For any such $k$, let $z_k$ be the unique solution in $z$ to
$$
r_1(z,k)=\log 2,
$$
where $r_1(z,k)$ is the function defined in \eqref{def:r1}.

We proceed by contradiction by assuming that there exists an integer $n'$ such that $\omega(n') \in \{17,\dotso,43\}$ and for which \eqref{inequality2} is false. We may also assume that $n'$ realizes the maximum of the function $r$ and moreover that $n'$ is of the form \eqref{shape}. As in Theorem \ref{thm:1}, we have $n' \le z_k$ and one can verify that $u(z_k/n_k) \le 5$. Thus, the exact same method that we used in the proof of Theorem \ref{thm:1} leads to an upper bound for the exponent vector of $n'$ given by
$$
(5,3,2,2,\underbrace{1,\dotso,1}_{k-4}).
$$
One can then verify, using a computer, that neither of these finite number of possibilities leads to a number $n$ that does not satisfy \eqref{inequality2}, thus contradicting the existence of $n'$.

We can therefore assume that $2 \le k \le 16$. Since $z_2=3.25\dotso$, $z_3=36.12\dotso$ and $r(30) < \log 2$, we deduce that in the particular cases $k=2$ and $k=3$, there is no counterexample in integers $n$ of the form \eqref{shape} to inequality \eqref{inequality2}. Thus, there is no counterexample in integers $n \ge 2$ with $\omega(n) \le 3$. For $4 \le k \le 16$ there are counterexamples to \eqref{inequality2} and thus we need to focus our attention on getting a good upper bound for every such integer in terms of $k$ only. In order to do this, we first exhibit the values of $u_k:=u(z_k/n_k)$ (easily obtained using a computer) in Table $1$.
\begin{center}
\begin{tabular}{|c||c|c|c|c|c|c|c|c|c|c|c|c|c|}\hline
$k$ & 4 & 5 & 6 & 7 & 8 & 9 & 10 & 11 & 12 & 13 & 14 & 15 & 16 \\ \hline
$u_k$ & 1 & 2 & 3 & 3 & 3 & 4 & 4 & 4 & 4 & 4 & 5 & 5 & 5 \\ \hline
\end{tabular}
\vskip 5pt
{\sc Table 1}
\end{center}
We can use this information to obtain an upper bound for $\tau(n)$ for any such counterexample $n$ of \eqref{inequality2}. Indeed, by using the multiplicativity of the function $\tau$ and inequality \eqref{fond1}, we get that for any such $n$ with $\omega(n)=k$,
$$
\tau(n) \le d_k:=\frac{2^{k-u_k}\beta(n_{u_k})}{u_k^{u_k}}\left(\log\frac{z_kn^2_{u_k}}{n_k}\right)^{u_k}.
$$
A priori this inequality is valid only for integers $n$ of the form \eqref{shape}, but it is then clearly also true for any counterexample to \eqref{inequality2} since any general counterexample to \eqref{inequality2} has an associated counterexample of the type \eqref{shape} with the same exponent vector once the prime factors are properly ordered. We use this inequality in \eqref{r} and introduce the function
$$
r_2(z,k):=\frac{1}{k}\left(\log d_k-k\log\left(\frac{\log z}{k\log k}\right)\right).
$$
Now, let $z^\prime_k$ be the unique solution in $z$ to
$$
r_2(z,k)=\log 2.
$$
Since $\frac{d}{dz}r_2(z,k)<0$, we deduce that $z^\prime_k$ is an upper bound for the largest possible counterexample $n$ to \eqref{inequality2} with an hypothetic value of $\tau(n)$ equal to $d_k$; clearly this is the largest among those we find with any smaller value of $\tau(n)$. We then find, using a computer, that $z_k^\prime$ is smaller than $24n_{16}$  for each $k \in \{4,\dotso,15\}$.

For $k=16$, the situation is somewhat different. Instead, we verify by using $z^\prime_{16}$ that there are only three possible exponent vectors, namely
\begin{eqnarray}\nonumber
&(3,1,1,1,1,1,1,1,1,1,1,1,1,1,1,1)&\\ \label{pos}
&(3,2,1,1,1,1,1,1,1,1,1,1,1,1,1,1)&\\ \nonumber
&(4,2,1,1,1,1,1,1,1,1,1,1,1,1,1,1)& \nonumber
\end{eqnarray}
that yield a counterexample to \eqref{inequality2} in integers $n$ of the type \eqref{shape}. For each of these, the smallest number strictly larger than the basic form is obtained by replacing the largest prime factor $p_{16}=53$ by $59$. We then obtain numbers $n$ which give $r(n)<\log 2$. We deduce that $24n_{16}$ (which corresponds to the last exponent vector in \eqref{pos}) is the largest of these. The proof of Theorem \ref{thm:2} is then complete.

\section{Proof of Theorem \ref{thm:3}}

We first verify that for $n_{*}:=720n_7$ we have $\lambda(n_{*})=1.1999953\dotso:=\eta_3$. We will show that $n_{*}$ is the only integer that maximizes $\lambda$. In order to reach a contradiction, we will assume that there exists $n' \neq n_{*}$ for which $\lambda(n') \ge \lambda(n_{*})$. Again, it is clear that the maximal value of $\lambda$ exists and is attained by an integer of the form \eqref{shape}. Therefore we will assume that $n'$ is of this form with $\omega(n')=k$. From \eqref{ineq5} and \eqref{ineq2}, it follows that the inequality
$$
\lambda(n') \le 1+\frac{\sum_{i=1}^{k}\log p_i-k\log k}{\log n'}
$$
is valid for each $k\ge 44$. On the other hand, we cannot have
\begin{equation}\label{n-sup}
\frac{\sum_{i=1}^{k}\log p_i-k\log k }{\log n'  }>\eta_3-1
\end{equation}
if $k \ge 44$, the reason being that since $n'$ has $k$ prime factors, it must satisfy $\log n' \ge \log n_k = \sum_{i=1}^{k}\log p_i$ in which case \eqref{n-sup} would imply
\begin{equation}\label{cont}
(2-\eta_3) \sum_{i=1}^{k}\log p_i>k\log k.
\end{equation}
But, using \eqref{ineq3}, it is easy to verify that \eqref{cont} is impossible when $k\ge 44$. This proves that we must have $k \le 43$. Considering \eqref{ineq5}, we let $z_k$ be the unique solution in $z$ of
$$
\upsilon(n_k,z)=\biggl(\prod_{p \mid n_k}\frac{\log k}{\log p}\biggr)^{1/k}+\frac{\log n_k}{\log z}\biggl(\prod_{p \mid n_k}\frac{\log k}{\log p}\biggr)^{1/k}-\frac{k\log k}{\log z}=\eta_3,
$$
where $\upsilon(n,z)$ is the function defined in \eqref{def:upsilon}. Since $\frac {d}{dz}\upsilon(n_k,z)<0$ by \eqref{eq:101}, we deduce that $n' \le z_k$. We find that the only possibilities for $n'$ are those with $k \in \{5,\dotso,13\}$, since otherwise we would have $n' \leq z_k < n_k$ which is impossible since by hypothesis we have $n_k \mid n'$.

Now, for $5\le k\le 13$ and from the fact that $n'$ is of the form \eqref{shape} with $\omega(n')=k$, we deduce that $n' = sn_k \le  z_k$ for some integer $s$ which satisfies $n_j|s$ with $j \le  k$. One can calculate that the largest ratio $z_k/n_k$ (for $5 \le k \le 13$) is less than $264\,507$. This forces $j\le 6$. Now, consider the set
$$
U:=\{s \le 264\,507:\ P(s) \le 13\},
$$
where $P(s)$ stands for the largest prime factor of $s$, and the set $V:=\{(sn_k,\lambda(sn_k)):\ s\in U\ \mbox{and}\ 5 \le k \le 13\}$. We observe that $V$ contains the element $(n_{*},r(n_{*}))$ and that for any other $n$ we have $r(n) < r(n_{*})$. This contradicts the existence of $n'$ and the proof of Theorem \ref{thm:3} is then complete.

\section{Proof of Theorem \ref{thm:4}}

In order to reach a contradiction, let us assume that there exists an integer $n'$ with $\omega(n')=k$, for some $k \ge 74$, for which \eqref{result} does not hold. For fixed values of $\omega(n)$ and $\tau(n)$, we see by definition \eqref{eq:def-lambda} that the function $\lambda(n)$ decreases as $n$ increases. For this reason, we will assume that $n'$ is of the form \eqref{shape}. We will also assume that $\lambda(n')$ is maximal.

For $k\ge 95$, we deduce from \eqref{ineq5}, \eqref{def:upsilon}, \eqref{eq:101} and \eqref{eq:102} that
$$
\lambda(n') \le \upsilon(n',n') \le \upsilon(n_k,n_k) < 1.
$$
This means that, inequality \eqref{result} holds for $k \ge 95$.

For each integer $k \in \{74,\dotso,94\}$, we cannot conclude since $\upsilon(n_k,n_k)>1$. However, since by Lemma \ref{lem:upsilon} we have $\frac {d}{dz}\upsilon(n_k,z) \le 0$ and $\upsilon(n_k,\infty)<1$, we can define $z_k$ implicitly by $\upsilon(n_k,z_k)=1$, in which case $n' \le z_k$. Also, observe that
$$
\log z_k/\log n_k<2
$$
for each $k$. This last inequality implies that the largest prime factor of $n'$ has its corresponding exponent equal to $1$.

As we have already seen, $u_k := u(z_k/n_k)$ provides an upper bound for the rank $j$ of the largest prime $p_j$ such that $p_j^2\mid n'$ since $n'$ is of the form \eqref{shape}. Our goal from now on is to verify all the remaining possibilities. To do so, we proceed in four steps. In the first step, we introduce a variable $j_1$ that will take the values $0,1,\dots,u_k$ and a variable $j_2$ that will take the values $0,1,\dots,\min(j_1,u_k)$. Then, we assume that
\begin{equation}\label{shape1}
n'=p_1^{\alpha_1}\cdots p_{j_2}^{\alpha_{j_2}}\cdot p_{j_2+1}^2\cdots p_{j_1}^2\cdot p_{j_1+1}\cdots p_k
\end{equation}
for some integers $\alpha_i \ge 3$ and that $n'$ is of the form \eqref{shape}. Now, if $0 < j_2 \le j_1$, by using the multiplicativity of the function $\tau$ together with \eqref{ineq5}, recalling the definition of $\lambda$ in \eqref{eq:def-lambda}, we are lead to consider the function
$$
f_1(j_2,j_1,k,z):=\frac{(c_1(j_2,j_1,k)(\log z - c_2(j_2,j_1,k))^{j_2/k}-1)k\log k }{\log z   }
$$
where
$$
c_1=c_1(j_2,j_1,k):=2^{(k-j_1)/k}3^{(j_1-j_2)/k}\frac{1}{j_2^{j_2/k}}\beta(n_{j_2})^{1/k}
$$
and
$$
c_2=c_2(j_2,j_1,k):=\log(n_kn_{j_1}/(n_{j_2})^3).
$$
Assume for now that each constant $c_2$ that will be considered through this proof satisfies
\begin{equation} \label{cond:c2}
c_2 > 0.
\end{equation}
Then, using Lemma \ref{lem:5}, we have
$$
\lambda(n') \le f_1(j_2,j_1,k,n') \le \max_{z > c_2(j_2,j_1,k)}f_1(j_2,j_1,k,z).
$$
Therefore, we will get the desired contradiction if $n'$ is of the type \eqref{shape1} and the unique maximum of $f_1(j_2,j_1,k,z)$ is proven to be less than $1$ (see Remark \ref{rem:3} for more details). The cases with $j_2=0$ or $j_1=0$ must be verified directly.

The values of $u_k$ are recorded in Table $2$.
\begin{center}
\begin{tabular}{|c||c|c|c|c|c|c|c|c|c|c|c|}\hline
$k$ & 74 & 75 & 76 & 77 & 78 & 79 & 80 & 81 & 82 & 83 & 84 \\ \hline
$u_k$ & 45 & 43 & 41 & 39 & 37 & 35 & 33 & 30 & 29 & 26 & 25 \\ \hline
\end{tabular}
\begin{tabular}{|c||c|c|c|c|c|c|c|c|c|c|}\hline
$k$ & 85 & 86 & 87 & 88 & 89 & 90 & 91 & 92 & 93 & 94 \\ \hline
$u_k$ & 23 & 21 & 19 & 17 & 15 & 13 & 11 & 8 & 6 & 2 \\ \hline
\end{tabular}
\vskip 5pt
{\sc Table 2}
\end{center}
All the computations being done, one is left with a reduced set of possibilities for the form of $n'$. In fact, we now have that $k \in \{74,75,76,77\}$ and also that the number of values that $j_1$ can take is significantly reduced. The final result is given in Table $3$.
\begin{center}
\begin{tabular}{|c||c|c|c|c|}\hline
$k$ & 74 & 75 & 76 & 77\\ \hline
$j_1 \in$ & \{14,\dotso,28\} & \{16,\dotso,26\} & \{18,\dotso,25\} & \{20,\dotso,23\} \\ \hline
\end{tabular}
\vskip 5pt
{\sc Table 3}
\end{center}
What we mean here is that a fixed pair $(k,j_1)$ is not in Table $3$ if for all $j_2 \le \min(j_1,u_k)$ we have $\displaystyle\max_{z > c_2(j_2,j_1,k)}f_1(j_2,j_1,k,z) < 1$.

This is where the second step of verifications starts. We now assume that
$$
n'=p_1^{\alpha_1}\cdots p_{j_3}^{\alpha_{j_3}}\cdot p_{j_3+1}^3\cdots p_{j_2}^3\cdot p_{j_2+1}^2\cdots p_{j_1}^2\cdot p_{j_1+1}\cdots p_k
$$
for some integers $\alpha_i \ge 4$ and we use the same argument as before to define the well suited function
$$
f_2(j_3,j_2,j_1,k,z):=\frac{(c_1(j_3,j_2,j_1,k)(\log z-c_2(j_3,j_2,j_1,k))^{j_3/k}-1)k\log k }{\log z},
$$
where
$$
c_1(j_3,j_2,j_1,k):=2^{(k-j_1)/k}3^{(j_1-j_2)/k}4^{(j_2-j_3)/k}\frac{1}{j_3^{j_3/k}}\beta(n_{j_3})^{1/k}
$$
and
$$
c_2(j_3,j_2,j_1,k):=\log(n_kn_{j_1}n_{j_2}/(n_{j_3})^4).
$$
We still have the chain of inequalities
$$
\lambda(n') \le f_2(j_3,j_2,j_1,k,n') \le \max_{z > c_2(j_3,j_2,j_1,k)}f_2(j_3,j_2,j_1,k,z).
$$
This time, we run this over the remaining values of $j_1$, and for
$$
j_2 \in \{1,\dotso,\min(j_1,u(z_k/(n_{j_1}n_k)))\}
$$
and
$$
j_3 \in \{1,\dotso,\min(j_2,u(z_k/(n_{j_2}n_{j_1}n_k)))\}.
$$
The cases with $j_3=0$ must be treated separately. Once again, these computations lead to further progress. We record in Table $4$ the remaining values which need to be examined.
\begin{center}
\begin{tabular}{|c||c|c|c|}\hline
$k$ & 74 & 75 & 76 \\ \hline
$j_1 \in$ & \{14,\dotso,23\} & \{16,\dotso,21\} & \{18,19\} \\ \hline
\end{tabular}
\vskip 5pt
{\sc Table 4}
\end{center}
We are now ready to begin the third step of verifications. We assume that
$$
n'=p_1^{\alpha_1}\cdots p_{j_4}^{\alpha_{j_4}}\cdot p_{j_4+1}^4\cdots p_{j_3}^4\cdot p_{j_3+1}^3\cdots p_{j_2}^3\cdot p_{j_2+1}^2\cdots p_{j_1}^2\cdot p_{j_1+1}\cdots p_k
$$
for some integers $\alpha_i \ge 5$ and define the function $f_3(j_4,j_3,j_2,j_1,k,z)$ in a similar manner by using the same ideas. However, we do introduce a new idea in the way of reducing the number of values that the variables $j_s$ $(s \ge 3)$ can take. We first assume that $p^\alpha \| n'$ for a fixed $\alpha \ge 2$, then we use \eqref{ratio1}, the fact that $n' \le z_k$ and the maximality of $\lambda(n')$ in order to write
$$
1 \le \frac{\lambda(n')}{\lambda(n'/p)} \le \left(1+\frac{2}{k\alpha}\right)\left(1-\frac{\log p}{\log n'}\right) \le
 \left(1+\frac{2}{k\alpha}\right)\left(1-\frac{\log p}{\log z_k}\right).
$$
We find a contradiction if $p$ is large enough to force the last expression to be less than $1$. In particular, we get an upper bound for the rank $j$ of such a prime $p_j$. Since this upper bound decreases when $\alpha$ increases, we obtain an upper bound for the rank $j$ of any prime $p_j$ for which $p_j^\alpha \mid n'$. Thus, by using \eqref{ratio3}, we obtain Table $5$.
\begin{center}
\begin{tabular}{|c||c|c|c|}\hline
$\alpha\diagdown k$ & 74 & 75 & 76 \\ \hline\hline
$ 4 $ & 11 & 11 & 10 \\ \hline
$ 5 $ & 7 & 6 & 6 \\ \hline
$ 6 $ & 4 & 4 & 4 \\ \hline
\end{tabular}
\vskip 5pt
{\sc Table 5}
\end{center}
Using this, we let $j_1$ take the values in Table $4$, and we let
$$
j_2 \in \{1,\dotso,\min(j_1,u(z_k/(n_{j_1}n_k)))\},
$$
$$
j_3 \in \{1,\dotso,\min(j_2,u(z_k/(n_{j_2}n_{j_1}n_k)),11\ \mbox{or}\ 10)\}
$$
and
$$
j_4 \in \{1,\dotso,\min(j_3,u(z_k/(n_{j_3}n_{j_2}n_{j_1}n_k)),7\ \mbox{or}\ 6)\}.
$$
Again, we treat the cases with $j_4=0$ independently. The computations lead to the result that we must have $k=74$ and $j_1 \in \{16,17,18\}$. We rule out these cases by defining  $f_4(j_5,j_4,j_3,j_2,j_1,k,z)$ and by using Table $5$ to limit the range of the variables $j_3,\ j_4$ and $j_5$. This completes the verifications.

It remains to prove \eqref{cond:c2}. To do so, we use the fact that $j_{s+1} \le j_s$ and that
$$
n_{j_s} \le \frac{z_k}{n_kn_{j_1}\cdots n_{j_{s-1}}}\quad\mbox{and}\quad -c_2(j_s,\dotso,j_1,k)=\log n_{j_s}+\log\frac{n_{j_s}^s}{n_kn_{j_1}\cdots n_{j_{s-1}}},
$$
from which we deduce that
\begin{eqnarray*}
-c_2(j_s,\dotso,j_1,k) & \le &  \log\frac{z_k}{n_kn_{j_1}\cdots n_{j_{s-1}}}+\log\frac{n_{j_s}^s}{n_kn_{j_1}\cdots n_{j_{s-1}}}\\
& \le & \log\frac{z_k}{n_kn_{j_1}}+\log\frac{n_{j_s}}{n_k}\\
& \le & \log\frac{z_k}{n_k^2} < -159.6
\end{eqnarray*}
by direct computation, which proves \eqref{cond:c2}.

We also observe that
\begin{eqnarray*}
-c_2(j_s,\dotso,j_1,k) & = & \log\frac{n_{j_s}^{s+1}}{n_kn_{j_1}\cdots n_{j_{s-1}}} > \log\frac{1}{n_kn_{j_1}\cdots n_{j_{s-1}}}\\
& > & \log\frac{1}{n_k^s} > -5\log n_{94} > -2342.
\end{eqnarray*}
This completes the proof of Theorem \ref{thm:4}.

\begin{remark}\label{rem:3}
We now provide some key details concerning the computations used in the proof of Theorem \ref{thm:4}. The information provided through the previous proof may differ with other information obtained with another strategy. We used $50$ decimals of precision for all computations. We used the criterion $f_s(\dotso) < 0.999999$ (for $s=1,2,3$ or $4$) for each comparison in the four steps of the computation and we kept a pair $(k,j_1)$ if we found $f_s(\dotso) \ge 0.999999$ somewhere in the process. By considering the function $h$ defined in \eqref{def:h}, we approximated $c_1$ and $c_2$ with $50$ decimals and we called $c_1^\prime$ and $c_2^\prime$ these approximations. Then we solved for $z_1$ in $h(z_1)=0$. From
$$
0=c_1^\prime\alpha z_1-c_1^\prime(z_1-c_2^\prime)+(z_1-c_2^\prime)^{1-\alpha} > c_1^\prime\alpha z_1-c_1^\prime(z_1-c_2^\prime),
$$
$\alpha \ge 1/94$ and $-c_2^\prime < -159$, we deduced that $z_1-c_2^\prime > 1.61$. The same is true also for the solution $z$ of
$$
0 = c_1\alpha z-c_1(z-c_2)+(z-c_2)^{1-\alpha}.
$$
It is easy to see that we always have $c_1 < 6!/\log 2 < 1039$. With this information at hand and using the mean value theorem, one finds that
$$
\left|\frac{c_1^\prime(z_1-c_2^\prime)^\alpha-1}{z_1}-\frac{c_1(z_1-c_2)^\alpha-1}{z_1}\right| < 10^{-43},
$$
thus concluding that the two functions are of about the same size for all values of $z$ or $z_1$ such that $z_1-c_2^\prime > 1.6$ or $z-c_2 > 1.6$. From the fact that $94\log 94 < 428$ and that an error of about $z_1\cdot 10^{-50}$ on $z_1$ cost less than $10^{-45}$ in the evaluation of $\frac{c_1^\prime(z_1-c_2^\prime)^\alpha-1}{z_1}$, we end up with an error of at most $10^{-40}$. This is small enough for the criterion we used.
\end{remark}

\section{Proof of Theorem \ref{thm:5}}

First, we verify that the integer $n_{*}$ defined in the statement of the theorem satisfies $\lambda(n_{*})>1$ and is of size $\exp(10640.8428\dots)$. Then, we claim that $n_{*}$ is the largest integer $n$ with $\omega(n) \ge 44$ and $\lambda(n) \ge 1$. To do so, we proceed by contradiction and assume that there exists an integer $n'$ such that $n' > n_{*}$ with $\lambda(n') \ge 1$ and $\omega(n') \ge 44$. The argument is done in several steps.

\subsection{Preliminary steps}

The first step consists in showing that we must have $\omega(n')=44$. For this, we use \eqref{ineq5}, \eqref{def:upsilon} and \eqref{eq:101} to deduce that if we define $z_k$ by $\upsilon(n_k,z_k) = 1$ then we must have $n' \le z_k$. We verify that $z_k \le \exp(4569.68) < n_{*}$ for each $k \in \{45,\dots,73\}$ and then we conclude using Theorem \ref{thm:4}.

We then want to show that $\gamma(n') = n_{44}$. This is done in two steps. We first assume that $n'$ is made of a choice of a set $S$ of $44$ distinct primes in $\{p_1,\dots,p_{45}\}$ and that this choice is not $S^\prime:=\{p_1,\dots,p_{44}\}$. There are $44$ possibilities and if we write $\displaystyle n_S:=\prod_{p \in S}p$, then using again the same argument as previously, we define $z_S$ by $\upsilon(n_S,z_S)=1$ and verify that $z_S \le \exp(9927.67) < n_{*}$ for each $S$.

Now, we assume that $n'$ has a general set of prime factors which has not been previously considered (and is not $S^\prime$), implying that there exists an integer $n'' < n'$ such that $\tau(n'')=\tau(n')$ and such that the set of prime divisors of $n''$ is $S$ for some $S \neq S^\prime$. We then have
$$
\lambda(n') < \lambda(n'') < 1
$$
if $n'' > z_S$. This proves that the set of prime factors of $n'$ must be $S^\prime$.

We solve for $z$ in the equation $\upsilon(n_{44},z)=1$ to find that $n' < \exp(10758.21)$. We have thus proved that
$$
10640.8 < \log n' < 10758.8.
$$
Consider the intervals $I_j:=[10639.8+j,10640.8+j]$ for each $j=1,\dots,118$. From now on, we want to show that $\log n'$ cannot be in any of these $I_j$.

\subsection{A first argumentation}

Recall the notation in \eqref{x} and \eqref{def:mu}, that is $x_i$ ($i=1,\dots,44$), $\mu$, $\mu_1$, $\mu_2$, $\varpi$ and $\varpi'$. The first argument that we use to eliminate some intervals $I_j$ relies on the inequality \eqref{je2} and on the proof of Corollary \ref{cor:1}. For a value of $m \in \{1,\dots,43\}$, we have
\begin{eqnarray*}
\tau(n') & \le & \log^{44} n' \beta(n_{44})\mu_1^m\mu_2^{44-m}\\
& = & {\scriptstyle \beta(n_{44})\left(\frac{\log n'}{44}\right)^{44}\left(1+\frac{\log n_{44}}{\log n'}-\frac{\varpi^\prime}{2m}\right)^{m}\left(1+\frac{\log n_{44}}{\log n'}+\frac{\varpi^\prime}{2(44-m)}\right)^{44-m}}
\end{eqnarray*}
so that if we write

\begin{eqnarray}\nonumber
{\scriptstyle\upsilon_{m}(z,w)} & {\scriptstyle :=} & {\scriptstyle\beta(n_{44})^{\frac{1}{44}}\log 44\left(1+\frac{\log n_{44}}{z}-\frac{w}{2m}\right)^{\frac{m}{44}}\left(1+\frac{\log n_{44}}{z}+\frac{w}{2(44-m)}\right)^{1-\frac{m}{44}}}\\ \label{def:upsilon1}
&& {\scriptstyle-\frac{44\log 44}{z}}
\end{eqnarray}
then we have
$$
\lambda(n') \le \max_{m \in \{1,\dots,43\}} \upsilon_{m}(\log n',\varpi^\prime).
$$
Thus, we define $z_{m,\varpi}$ by $\upsilon_{m}(z_{m,\varpi},\varpi)=1$. We have seen that $z_{m,0}=10758.2\dots$ From {\bf (i)} and {\bf (iii)} of Lemma \ref{lem:02}, we know that $z_{m,w}$ decreases when $w$ increases. We record in Table $6$ a value of $w:=w(j)$ such that
$$
\max_{m \in \{1,\dots,43\}} z_{m,w(j)} < 10639.8+j
$$
for the first $39$ values of $j$.

\begin{center}
\begin{tabular}{|c||c|c|c|c|c|c|c|c|}\hline
$j$ & 1 & 2 & 3 & 4 & 5 & 6 & 7 & 8 \\ \hline
$w$ & 0.2137 & 0.2128 & 0.2119 & 0.2109 & 0.2100 & 0.2091 & 0.2081 & 0.2072 \\ \hline
\end{tabular}
\begin{tabular}{|c||c|c|c|c|c|c|c|c|}\hline
$j$ & 9 & 10 & 11 & 12 & 13 & 14 & 15 & 16 \\ \hline
$w$ & 0.2062 & 0.2053 & 0.2043 & 0.2034 & 0.2024 & 0.2014 & 0.2004 & 0.1995 \\ \hline
\end{tabular}
\begin{tabular}{|c||c|c|c|c|c|c|c|c|}\hline
$j$ & 17 & 18 & 19 & 20 & 21 & 22 & 23 & 24 \\ \hline
$w$ & 0.1985 & 0.1975 & 0.1965 & 0.1955 & 0.1945 & 0.1935 & 0.1925 & 0.1914 \\ \hline
\end{tabular}
\begin{tabular}{|c||c|c|c|c|c|c|c|c|}\hline
$j$ & 25 & 26 & 27 & 28 & 29 & 30 & 31 & 32 \\ \hline
$w$ & 0.1904 & 0.1894 & 0.1884 & 0.1873 & 0.1863 & 0.1852 & 0.1841 & 0.1831 \\ \hline
\end{tabular}
\begin{tabular}{|c||c|c|c|c|c|c|c|}\hline
$j$ & 33 & 34 & 35 & 36 & 37 & 38 & 39 \\ \hline
$w$ & 0.1820 & 0.1809 & 0.1799 & 0.1788 & 0.1777 & 0.1766 & 0.1755 \\ \hline
\end{tabular}
\vskip 5pt
{\sc Table 6}
\end{center}

In the opposite direction, a lower bound for $\varpi^\prime(=\varpi^\prime(n'))$ can be computed for $n'$ assuming that $\log n'$ is in $I_j$. To do so, we split the interval $I_j$ in $210$ subintervals of length $\frac{1}{210}$ that we call $I_{j,j_1}$ where $1 \le j_1 \le 210$. We use Lemma \ref{lem:03} {\bf (i)} with $\varphi=1$ term by term to compute
$$
\min_{z \in I_{j,j_1}}\sum_{i=1}^{44}\min_{\alpha_i \in \Z}\left|\frac{(\alpha_i+1)44\log p_i-\log n_{44}}{z}-1\right|
$$
and take the minimum over the variable $j_1$ to get the lower bound for $\varpi^\prime(n')$ for $\log n'$ in $I_j$. We record the result in Table $7$.

\begin{center}
\begin{tabular}{|c||c|c|c|c|c|c|c|c|}\hline
$j$ & 1 & 2 & 3 & 4 & 5 & 6 & 7 & 8 \\ \hline
$\varpi^\prime$ & 0.1814 & 0.1812 & 0.1810 & 0.1808 & 0.1804 & 0.1802 & 0.1800 & 0.1798 \\ \hline
\end{tabular}
\begin{tabular}{|c||c|c|c|c|c|c|c|c|}\hline
$j$ & 9 & 10 & 11 & 12 & 13 & 14 & 15 & 16 \\ \hline
$\varpi^\prime$ & 0.1797 & 0.1797 & 0.1798 & 0.1800 & 0.1800 & 0.1800 & 0.1798 & 0.1797 \\ \hline
\end{tabular}
\begin{tabular}{|c||c|c|c|c|c|c|c|c|}\hline
$j$ & 17 & 18 & 19 & 20 & 21 & 22 & 23 & 24 \\ \hline
$\varpi^\prime$ & 0.1797 & 0.1798 & 0.1800 & 0.1800 & 0.1800 & 0.1798 & 0.1796 & 0.1792 \\ \hline
\end{tabular}
\begin{tabular}{|c||c|c|c|c|c|c|c|c|}\hline
$j$ & 25 & 26 & 27 & 28 & 29 & 30 & 31 & 32 \\ \hline
$\varpi^\prime$ & 0.1788 & 0.1784 & 0.1781 & 0.1779 & 0.1777 & 0.1775 & 0.1773 & 0.1771 \\ \hline
\end{tabular}
\begin{tabular}{|c||c|c|c|c|c|c|c|}\hline
$j$ & 33 & 34 & 35 & 36 & 37 & 38 & 39 \\ \hline
$\varpi^\prime$ & 0.1769 & 0.1765 & 0.1763 & 0.1761 & 0.1755 & 0.1751 & 0.1748 \\ \hline
\end{tabular}
\vskip 5pt
{\sc Table 7}
\end{center}

Also, we verify that for each $j \in \{40,\dots,118\}$ we have $\varpi^\prime(j)>w(j)$, thereby implying that there exist no $n'$ with $\log n'$ in $I_j$.

All of this gives rise to a new concept that will be crucial for the remaining of the proof. This is the difference between the upper and lower bounds for $\varpi^\prime$. Just saying that here the difference when $j=1$, that is $0.2137-0.1814=0.0323$, is too large for us. In fact, it will be convenient to work with a slightly different concept. Consider the function defined on primes $p$ by
\begin{equation}\label{def:eps}
\epsilon_j(p):=\min_{z \in I_j}\left|\frac{(\alpha+1)44\log p-\log n_{44}}{z}-1\right|.
\end{equation}
The value of $\epsilon_j(p)$ is computed by using Lemma \ref{lem:03} {\bf (i)}. For each $j \in \{1,\dots,39\}$ we sum the $\epsilon_{j}(p_i)$ for $i \in \{ 1,\dots,44\}$ and subtract the answer from the upper bound $w(j)$. We call these value $\delta^\prime(=\delta^\prime(j))$ and  record them in Table $8$.

\begin{center}
\begin{tabular}{|c||c|c|c|c|c|c|c|c|}\hline
$j$ & 1 & 2 & 3 & 4 & 5 & 6 & 7 \\ \hline
$\delta^\prime$ & 0.03422 & 0.03353 & 0.03283 & 0.03203 & 0.03142 & 0.03083 & 0.03005 \\ \hline
\end{tabular}
\begin{tabular}{|c||c|c|c|c|c|c|c|c|}\hline
$j$ & 8 & 9 & 10 & 11 & 12 & 13 & 14 \\ \hline
$\delta^\prime$ & 0.02936 & 0.02848 & 0.02753 & 0.02638 & 0.02536 & 0.02436 & 0.02340 \\ \hline
\end{tabular}
\begin{tabular}{|c||c|c|c|c|c|c|c|c|}\hline
$j$ & 15 & 16 & 17 & 18 & 19 & 20 & 21 \\ \hline
$\delta^\prime$ & 0.02253 & 0.02178 & 0.02075 & 0.01958 & 0.01846 & 0.01748 & 0.01650 \\ \hline
\end{tabular}
\begin{tabular}{|c||c|c|c|c|c|c|c|c|}\hline
$j$ & 22 & 23 & 24 & 25 & 26 & 27 & 28 \\ \hline
$\delta^\prime$ & 0.01561 & 0.01481 & 0.01398 & 0.01340 & 0.01279 & 0.01216 & 0.01133 \\ \hline
\end{tabular}
\begin{tabular}{|c||c|c|c|c|c|c|c|}\hline
$j$ & 29 & 30 & 31 & 32 & 33 & 34 & 35 \\ \hline
$\delta^\prime$ & 0.01053 & 0.00964 & 0.00874 & 0.00795 & 0.00705 & 0.00618 & 0.00547 \\ \hline
\end{tabular}
\begin{tabular}{|c||c|c|c|c|}\hline
$j$ & 36 & 37 & 38 & 39 \\ \hline
$\delta^\prime$ & 0.00458 & 0.00392 & 0.00331 & 0.00258 \\ \hline
\end{tabular}
\vskip 5pt
{\sc Table 8}
\end{center}
The value $\delta^\prime$ is to be interpreted as an upper bound to the extra error that can produce $n'$.

\subsection{A first verification}

We want to make some direct verifications to prove that $\lambda(n') \ge 1$ is impossible if the exponent vector of $n'$ is of a certain type. Consider the sets
\begin{equation}\label{J}
{\textstyle
J_{\delta}(p,j):=\left\{\left\lceil\frac{(1-\delta)(10639.8+j)+\log n_{44}}{44\log p}\right\rceil-1,\dots,\left\lfloor\frac{(1+\delta)(10640.8+j)+\log n_{44}}{44\log p}\right\rfloor-1\right\}}.
\end{equation}
The set $J_\delta(p,j)$ has the property that if
$$
\left|\frac{(\alpha+1)44\log p-\log n_{44}}{\log n}-1\right| \le \delta\quad\mbox{with}\quad \log n \in I_j
$$
then $\alpha \in J_\delta(p,j)$.

We divide the verifications into two distinct types. Type $1$ concerns the sets
$$
S_j(\delta):=J_{\delta}(p_1,j) \times \cdots \times J_{\delta}(p_{44},j).
$$
We take $\delta=0.011$ for $j \in \{1,\dots,4\}$ and $\delta=0.01$ for $j \in \{5,\dots,14\}$. Also, to speed up the process, we consider the union term-by-term of $S_1(0.011),\dots,S_4(0.011)$ to get a new set $S_1$ say, so that $S_1=J_{0.011}(2,1) \cup \dots \cup J_{0.011}(2,4) \times \dots$ We do the same with $S_5(0.01),\dots,S_{14}(0.01)$ to get $S_2$. These sets have respectively $92160$ and $53760$ elements. For each vector $v=(\alpha_1,\dots,\alpha_{44})$ in each of these two sets, we take one of the $946$ possible choices of two elements in a set of $44$ elements, say $(i_1,i_2)$, and construct the new set
$$
{\scriptscriptstyle\{\alpha_1\}\times\cdots\times\{\alpha_{i_1-1}\}\times J_{\epsilon_j(p_{i_1})+\delta^\prime(j)}(p_{i_1},j)\times\{\alpha_{i_1+1}\}\times\cdots\times\{\alpha_{i_2-1}\}\times J_{\epsilon_j(p_{i_2})+\delta^\prime(j)}(p_{i_2},j)\times\{\alpha_{i_2+1}\}\times\cdots\times\{\alpha_{44}\}}.
$$
We verify that all these exponent vectors $v$ give rise to an integer $n$ such that $\lambda(n) < 1$, $\log n < 10640.8$ or is itself $n_{*}$.

Type $2$ concerns the sets
$$
S^\prime_j(\delta):=J_{\epsilon_j(p_1)+\delta}(p_1,j) \times \cdots \times J_{\epsilon_j(p_{44})+\delta}(p_{44},j).
$$
This time, we take $\delta=0.0055$ for $j \in \{1,\dots,6\}$, $\delta=0.0054$ for $j \in \{7,\dots,9\}$ and $j \in \{10,\dots,13\}$, $\delta=0.005$ and $j \in \{14,\dots,19\}$, $\delta=0.0044$ for $j \in \{20,\dots,23\}$, $\delta=0.004$ for $j \in \{24,25,26\}$, $\delta=0.0035$ for $j \in \{27,28,29\}$ and $\delta=0.003$ for $j \in \{30,\dots,39\}$. Again, to speed up the process, we consider the unions term-by-term the same way, so that we have $S_1^\prime=J_{\epsilon_1(2)+0.0055}(2,1) \cup \dots \cup J_{\epsilon_6(2)+0.0055}(2,6) \times \dots$ and the same for $S_2^\prime, \dots, S_8^\prime$. These sets have respectively $98304$, $73728$, $49152$, $49152$, $32768$, $32768$, $32768$ and $24576$ elements. For each vector $v$, we do the exact same process as for the type $1$.

At the end of these verifications, we know that there are at least three entries in the exponent vector that produce a large error and this occurs in both type $1$ and $2$.

\subsection{Reducing the upper bound for $\delta^\prime$}

Our strategy begins with a lower bound for $\varpi^\prime_1$ and $\varpi^\prime_2$. For each $j$, there are four cases to consider depending on the position of the $x_i$ \eqref{x} compared to $\mu$ \eqref{def:mu}. Indeed, we have seen in the previous section, with the type $1$ verification, that there are at least three $x_i$ that are far from $\mu$ but this does not tell us where they are. So any lower bound for $\varpi^\prime_1$ and $\varpi^\prime_2$ will come in pair $(\varpi^\prime_1,\varpi^\prime_2)$ with the total number $m$ of $x_i$ that are less than or equal to $\mu$ and with a position signature $s$ in $\{0,1,2,3\}$ that tells us that the number of $x_i$ that are less than $\mu$ in these three we assume to have. This number $m$ can be shown to take the values we recorded in Table $9$.
\begin{center}
\begin{tabular}{|c||c|c|c|c|c|}\hline
$j \in$ & \{1,\dots,6\} & \{7\} & \{8,\dots,12\} & \{13\} & \{14\} \\ \hline
$m \in$ & \{11,\dots,33\} & \{12,\dots,33\} & \{12,\dots,32\} & \{13,\dots,32\} & \{13,\dots,31\}\\ \hline
\end{tabular}
\vskip 5pt
{\sc Table 9}
\end{center}
To do so, we use Table $7$ and verify that $z_{m,\varpi^\prime(j)}< 10639.8+j$ (see section $9.2$) for all the values of $m$ not listed in Table $9$.

Also, the definition of $\varpi^\prime_1$ and $\varpi^\prime_2$ in \eqref{w1} and \eqref{w2} includes the exact value of $\varpi^\prime$, something that we cannot know precisely. So we assume an interval containing the value of $\varpi^\prime$, and look for a contradiction. More precisely, we will assume, for $j \in \{1,\dots,14\}$, that $\varpi^\prime$ belongs to $W(j):=[w(j)-0.01,w(j)]$.

To get these lower bounds, we fix $j$, $m$ and a signature $s$. Then, we split the interval $I_j$ in $30$ subintervals $I_{j,r_1}$ ($r_1=1,\dots,30$) of equal length and we split the interval $W(j)$ in $80$ subintervals $W(j,r_2)$ ($r_2=1,\dots,80$) of equal length.

We fix $I_{j,r_1}$ and $W(j,r_2)$ and begin with $\varpi^\prime_1$. At first, we focus on the $s$ points $x_i$ that are less than $\mu$. We will show that in this case the minimum of
\begin{equation}\label{ff}
\left|\frac{(\alpha+1)44\log p-\log n_{44}}{x}-1+\frac{\varpi^\prime}{2m}\right|
\end{equation}
is attained with $\alpha:=\left\lceil\frac{(1-\delta)(10639.8+j)+\log n_{44}}{44\log p}\right\rceil-2$ and, as in Lemma \ref{lem:03} {\bf (ii)}, at the extremity of the intervals $I_{j,r_1}$ and $W(j,r_2)$. In fact, from the definition of $J_\delta$, it is enough to show that $\frac{(\alpha+1)44\log p-\log n_{44}}{x}-1+\frac{\varpi^\prime}{2m} < 0$ and this follows from
{\tiny\begin{eqnarray*}
\frac{(\alpha+1)44\log p-\log n_{44}}{x}-1+\frac{\varpi^\prime}{2m} & = & \frac{\left(\left\lceil\frac{(1-\delta)(10639.8+j)+\log n_{44}}{44\log p}\right\rceil-1\right)44\log p-\log n_{44}}{x}-1+\frac{\varpi^\prime}{2m} \\
& = & \frac{\left(\frac{(1-\delta)(10639.8+j)+\log n_{44}}{44\log p}+\xi-1\right)44\log p-\log n_{44}}{x}-1+\frac{\varpi^\prime}{2m} \\
& = & \frac{(1-\delta)(10639.8+j)}{x}-\frac{(1-\xi)44\log p}{x}-1+\frac{\varpi^\prime}{2m} \\
& \le & -\delta+\frac{\varpi^\prime}{2m} < 0
\end{eqnarray*}}
since $2m\delta > 2 > \varpi^\prime$ from our choices, where $0 \le \xi < 1$ and both $\varpi^\prime$ and $x$ are seen as fixed. We keep the $s$ smallest such values among the $44$ prime numbers.

Then, we compute the minimum value of \eqref{ff}, without any constraint on $\alpha$, using Lemma \ref{lem:03} {\bf (i)}. We keep the $m-s$ smallest ones among the $44$ prime numbers. We sum the $m$ values we have kept so far and we take the minimum among the $30 \cdot 80 = 2400$ possible values of $(r_1,r_2)$ and this is the wanted lower bound for $\varpi^\prime_1$ in $I(j)$ with this value of $m$ and $s$. We do the same for $\varpi^\prime_2$ with $3-s$ values of $x_i$ greater than $\mu$ along with the choice $\alpha:=\left\lfloor\frac{(1+\delta)(10640.8+j)+\log n_{44}}{44\log p}\right\rfloor$ instead and the function
\begin{equation}\label{ff2}
\left|\frac{(\alpha+1)44\log p-\log n_{44}}{x}-1-\frac{\varpi^\prime}{2(44-m)}\right|
\end{equation}
as in the definition of $\varpi^\prime_2$ in \eqref{w2}. The proof is similar. We obtain the value for $\varpi^\prime_2$ in $I_j$ with parameters $3-s$ and $44-m$ instead. We keep the pair $(\varpi^\prime_1,\varpi^\prime_2)(=(\varpi^\prime_1(j,m,s),\varpi^\prime_2(j,m,s)))$. For example, here is the output we get as lower bound with $j=1$ and $m=11$:
$$
(0.010296421544, 0.093154520284), (0.010296421544, 0.089438737225),
$$
$$
(0.011179764497, 0.087104430865), (0.012637967643, 0.085223479629)
$$
for $(\varpi^\prime_1,\varpi^\prime_2)$ when $s=0,1,2,3$ respectively. Note that these values are just stated as an example, they are sensitive to the way the program is done. Now, using the same reasoning as we did to get to \eqref{def:upsilon1}, we are led to consider
\begin{eqnarray}\nonumber
{\scriptstyle \upsilon_{m,m_1,m_2}(z,w,w_1,w_2)} & {\scriptstyle :=} & {\scriptstyle \beta(n_{44})^{\frac{1}{44}}\log 44\left(1+\frac{\log n_{44}}{z}-\frac{w}{2m}-\frac{w_1}{2m_1}\right)^{\frac{m_1}{44}}} \\ \nonumber
&& {\scriptstyle \cdot \left(1+\frac{\log n_{44}}{z}-\frac{w}{2m}+\frac{w_1}{2(m-m_1)}\right)^{\frac{m-m_1}{44}}} \\ \nonumber
&& {\scriptstyle \cdot \left(1+\frac{\log n_{44}}{z}+\frac{w}{2(44-m)}-\frac{w_2}{2m_2}\right)^{\frac{m_2}{44}}} \\ \nonumber
&& {\scriptstyle \cdot \left(1+\frac{\log n_{44}}{z}+\frac{w}{2(44-m)}+\frac{w_2}{2(44-m-m_2)}\right)^{1-\frac{m+m_2}{44}}} \\ \label{def:upsilon2}
&& {\scriptstyle -\frac{44\log 44}{z}}
\end{eqnarray}
for fixed values of $m_1 \in \{1,\dots,m-1\}$ and $m_2 \in \{1,\dots,43-m\}$. We also define $z_{j,m,m_1,m_2,s}$ implicitly by
$$
\upsilon_{j,m,m_1,m_2,s}(z_{j,m,m_1,m_2,s},w(j)-0.01,\varpi^\prime_1(j,m,s),\varpi^\prime_2(j,m,s))=1.
$$
and verify that
$$
\max_{m}\max_{s \in \{0,\dots,3\}}\max_{m_1 \in \{1,\dots,m-1\}}\max_{m_2 \in \{1,\dots,43-m\}}z_{j,m,m_1,m_2,s} < 10639.8+j,
$$
where the maximum is taken over the values of $m$ appearing in Table $9$. We justify that $z_{j,m,m_1,m_2,s}$ is the appropriate choice by using Lemma \ref{lem:02} part {\bf (i), (ii)} and {\bf (iv)} (there is a condition to verify in part {\bf (ii)}). We find that $\upsilon_{m,m_1,m_2}$ would be smaller with larger values of the variables. This is the contradiction we were looking for. So we have in fact that $\varpi^\prime \notin W(j)$ and a new upper bound for $\delta^\prime$ recorded in Table $10$.

\begin{center}
\begin{tabular}{|c||c|c|c|c|c|c|c|}\hline
$j$ & 1 & 2 & 3 & 4 & 5 & 6 & 7 \\ \hline
$\delta^\prime$ & 0.02422 & 0.02353 & 0.02283 & 0.02203 & 0.02142 & 0.02083 & 0.02005 \\ \hline
\end{tabular}
\begin{tabular}{|c||c|c|c|c|c|c|c|}\hline
$j$ & 8 & 9 & 10 & 11 & 12 & 13 & 14 \\ \hline
$\delta^\prime$ & 0.01936 & 0.01848 & 0.01753 & 0.01638 & 0.01536 & 0.01436 & 0.01340\\ \hline
\end{tabular}
\vskip 5pt
{\sc Table 10}
\end{center}

\subsection{The last verification}

Our strategy of verification begins with a preliminary computation. We use the type $2$ computations that we did previously to prove that at least three points $x_i$ defined in \eqref{x} are far from $\mu$ defined in \eqref{def:mu}. We first want to show that, among the 13244 possibilities of triplets of primes, at most a few hundreds can produce these three values of $x_i$.

To do so, we fix $j$ and split the interval $I_j$ into $25$ subintervals $I_{j,r}$ of equal length. We also fix a triplet $(q_1,q_2,q_3)$. Now, the type $2$ computations reveal that the exponent of a prime $p \in \{q_1,q_2,q_3\}$ that divide exactly $n'$ is not in $J_{\delta+\epsilon_j(p)}(p,j)$ where $\delta=\delta(j)$ can be found in section $9.3$. We are thus in the exact situation of Lemma \ref{lem:03} {\bf (ii)}. So that we compute and sum the three minimal errors, we take the minimum over $r=1,\dots,25$ and call this minimum $\zeta(=\zeta(q_1,q_2,q_3))$. If
$$
\zeta-\epsilon_j(q_1)-\epsilon_j(q_2)-\epsilon_j(q_3) > \delta^\prime(j)
$$
then the triplet $(q_1,q_2,q_3)$ is rejected. Otherwise, we keep
$$
(q_1,q_2,q_3,\zeta(q_1,q_2,q_3)-\epsilon_j(q_1)-\epsilon_j(q_2)-\epsilon_j(q_3))
$$
to the last verification in a set $T(j)$, say. The value of $\delta^\prime$ is picked from in Table $10$ if $j \le 14$ and from Table $8$ if $15 \le j \le 39$.

Now, for the very last verification, after all the $T(j)$ have been computed, we use a new idea. We assume that $j$ is fixed. For a prime $p$ in a fixed vector $(q_1,q_2,q_3,\rho) \in T(j)$, we observe that it is enough to check the integers $n$ with the exponent in $J_{\delta^\prime(j)+\epsilon_j(p)}(p,j)$. Then, for the remaining $41$ primes $p$, it is enough to verify with the exponent in the set
$$
J_{\delta^\prime(j)/2-\rho/2+\epsilon_j(p)}(p,j)
$$
for all but one prime $p$ for which it can be in $J_{\delta^\prime(j)-\rho+\epsilon_j(p)}(p,j)$.

With these observations in mind, we design an algorithm. We compute the largest fourth component in any of the vectors in $T(j)$ and call it $t$. Then we consider only the vectors such that the fourth component is in $[t-u/1000,t-(u-1)/1000]$ for a fixed $u \in \{1,\dots,7\}$, which we denote by $T_u(j)$. With $u$ fixed, we store in memory all the vectors in
$$
J_{\delta^\prime(j)/2-(t-u/1000)/2+\epsilon_j(2)}(2,j) \times \dots \times J_{\delta^\prime(j)/2-(t-u/1000)/2+\epsilon_j(193)}(193,j)
$$
to which we add two dimensions: one of which is the value of $\tau(n)^{1/44}$ of the integer $n$ with this exponent vector whereas the other is its logarithm. Then, for all such vectors, only four exponents have to be modified at each verification and thus the last two informations need only a small adjustment to be used to compute the value of $\lambda$ in each case. So, for each vector of $46$ dimensions, for each exponent in $J_{\delta^\prime(j)+\epsilon_j(p)}(p,j)$  of each prime $p$ in each triplet (in a vector) in $T_u(j)$ and for each exponent in $J_{\delta^\prime(j)-\rho+\epsilon_j(p)}(p,j)$ of each other $41$ prime $p$ we compute the corresponding value of $\lambda$. We try each value of $u$ and then all the values of $j$.

After all these verifications, no value of $n'$ have been found. This is the contradiction we were searching for and thus $n_{*}$ is the largest number $n$ such that $\lambda(n)>1$ and $\omega(n) \ge 44$. The proof is complete.

\section{Final remarks}

One can show that
$$
\sum_{n\le  x}\left|\lambda(n)-\frac{\log\log x\log\log\log x}{\log x}\right|^2\ll\frac{x\log\log x (\log\log\log x)^2}{\log^2 x},
$$
from which we may conclude that for almost all $n \le  x$,
$$
\lambda(n)=(1+o(1))\frac{\log\log x\log\log\log x}{\log x}\quad(x\rightarrow\infty).
$$
On the other hand, we can show that there are infinitely many $n$ for which $\lambda(n) > 1$. Indeed, to any set $S$ of primes satisfying
$$
\prod_{p \in S}\frac{\log k}{\log p} > 1\quad\mbox{and}\quad \#S=k,
$$
we can associate a sequence of integers $l_1,\ l_2,\ \dotso$ such that their exponent at each prime factor, and then the associated $\theta_i$ as defined in Corollary \ref{cor:1}, is as close as possible to the optimal value as defined in Lemma \ref{ineqfond}. Precisely, for each $p' \in S$,we can choose $m_j$ to be an integer for which the exponent of $p'$, $\alpha_{p'}$, is the closest integer to $\frac{1}{k}\left(\frac{\log z_j}{\log p'}+\sum_{p \in S}\frac{\log p}{\log p'}\right)-1$ for a fixed large $z_j$. One verifies that
$$
\lambda(l_j) \rightarrow \biggl(\prod_{p \in S}\frac{\log k}{\log p}\biggr)^{1/k}\quad(z_j\rightarrow\infty).
$$

Finally, we can also show that the set of limit points of $\lambda(n)$ is the interval $[0,\beta(6)^{1/6}\log 6]$ $=[0,1.145206\dotso]$ and that there exists a positive constant $\eta$ such that
$$
\#\{n \le x|\ \lambda(n) \ge 1\} = (\eta+o(1))\log^{43}x\quad(x\rightarrow\infty).
$$
Moreover, we have
$$
\sup_{\omega(n)=k}\lambda(n)=1-\frac{\log\log k-1}{\log k}+\frac{(\log\log k)^2-3\log\log k}{\log^2k}+O\left(\frac{1}{\log^2k}\right)\quad(k\rightarrow\infty).
$$

\bibliographystyle{amsplain}

{\sc D\'epartement de math\'ematiques et de statistique, Universit\'e Laval, Pavillon Alexandre-Vachon, 1045 Avenue de la M\'edecine, Qu\'ebec, QC G1V 0A6} \\
{\it E-mail address:} {\tt jmdk@mat.ulaval.ca}
\vskip 10pt
{\sc D\'epartement de math\'ematiques et de statistique, Universit\'e Laval, Pavillon Alexandre-Vachon, 1045 Avenue de la M\'edecine, Qu\'ebec, QC G1V 0A6} \\
{\it E-mail address:} {\tt Patrick.Letendre.1@ulaval.ca}

\end{document}